\documentclass[12pt]{article}
\usepackage{amsmath,amssymb,amsfonts,amsthm,array,verbatim,xcolor}
\usepackage{subcaption}
\usepackage[margin=1in]{geometry}
  
\usepackage{graphicx}

\theoremstyle{plain}

\newtheorem{theorem}{Theorem}[section]
\newtheorem{lemma}[theorem]{Lemma}
\newtheorem{conjecture}[theorem]{Conjecture}

\newtheorem{corollary}[theorem]{Corollary}

\title{An optimal construction for complete graph embeddings with duals of low connectivity}
\author{Timothy Sun\\Department of Computer Science\\San Francisco State University}
\date{}

\newcommand{\Z}{\mathbb{Z}}
\newcommand{\RL}[1]{\emph{\textbf{\underline{#1}}}}
\newcommand{\SM}{\small}

\begin{document}

\maketitle

\begin{abstract}
We describe a construction for embeddings of complete graphs where the dual has a cutvertex and the genus is close to the minimum genus of the primal graph. When the number of vertices is congruent to 5 modulo 12, we further guarantee that the dual is simple and that the genera of the resulting embeddings match a lower bound of Brinkmann, Noguchi, and Van den Camp, showing that their lower bound is tight infinitely often. 
\end{abstract}

\section{Introduction}

In a recent pair of papers, Bokal, Brinkmann, Noguchi, Van den Camp, and Zamfirescu~\cite{BBZ, BNV} studied the relationships between the connectivities of the primal and dual graphs of embeddings in surfaces. In particular, given two integers $c$ and $k \in \{1,2\}$, they were interested in a parameter they denoted as $\delta_k(c)$, the smallest integer $\delta$ where there is a cellular embedding of a $c$-connected graph in the orientable surface of genus $\delta$ whose dual is simple and has a $k$-vertex-cut. For sufficiently large values of $c$, bounds on $\delta_k(c)$ are related to embeddings of the complete graph $K_{c+1}$, whose minimum genus is given by the formula
$$\gamma(K_{c+1}) = \left\lceil \frac{(c-2)(c-3)}{12}\right\rceil.$$
Brinkmann, Noguchi, and Van den Camp~\cite{BNV} proved the lower bounds $\delta_1(c) \geq \gamma(K_{c+1})+2$ and $\delta_2(c) \geq \gamma(K_{c+1})+1$. The latter bound matches a construction in Bokal, Brinkmann, and Zamfirescu~\cite{BBZ}, so we focus on the former bound. 

For brevity, we call an embedding of a $c$-connected graph with a simple dual of connectivity 1 a \emph{dual-separable} embedding, and if it has genus $\gamma(K_{c+1})+2$, we say that it is \emph{optimal}. Brinkmann et al.~\cite{BNV} showed that for some values of $c$, one can construct dual-separable embeddings that are nearly optimal. In particular, given a triangular embedding of $K_{c+1}-E(K_6)$, it can be modified into a dual-separable embedding of $K_{c+1}$ in the surface of genus $\gamma(K_{c+1})+3$. They then applied their construction to the triangular embeddings of $K_{12s+9}-E(K_6)$, $s \geq 2$, constructed in Sun~\cite{Sun-Minimum}. We note that such embeddings of $K_n-E(K_6)$ are also known for $n = 12s+10$, $s \geq 2$~\cite{Sun-Kainen}, and an embedding for $n = 18$ was announced by Jungerman~\cite{Jungerman-K18}.

The purpose of this note is to show that a known family of triangular embeddings of $K_{12s+5}-E(K_2)$ can be modified into optimal dual-separable embeddings:

\begin{theorem}
For $s \geq 1$, $\delta_1(12s+4) = \gamma(K_{12s+5})+2.$
\label{thm-main}
\end{theorem}

Our construction revolves around what Jungerman and Ringel~\cite{JungermanRingel-Minimal} called a ``subtractible handle,'' a small set of edges whose deletion decreases the genus by 1. We also sketch how the same construction can be applied to other complete graphs to obtain another nearly optimal result. Finally, some small optimal dual-separable embeddings are given in Appendix \ref{app-small}.  

\section{Background}

We assume familiarity with topological graph theory, especially the theory of current graphs. For more information, see Ringel~\cite{Ringel-MapColor} (in particular, Section 9) and Gross and Tucker~\cite{GrossTucker}. 

Let $G = (V,E)$ be a graph, and let $\phi\colon G \to S$ be an embedding of a graph $G$ in an orientable surface $S_g$ of genus $g$. In this paper, we restrict ourselves to embeddings that are cellular, i.e., ones where the connected components of $S \setminus \phi(G)$ are disks. The connected components are called \emph{faces} and are denoted by the set $F$. The number of faces is governed by Euler's polyhedral formula, which states that
$$|V|-|E|+|F|=2-2g.$$

Each edge $e$ induces two arcs $e^+$ and $e^-$ pointing in opposite directions. We use $E^+$ to denote the set of such arcs. A \emph{rotation} at a vertex is a cyclic permutation of the arcs leaving that vertex, and a \emph{rotation system} is an assignment of a rotation to each vertex of the graph. For a simple graph, a rotation can be specified by simply a cyclic permutation of the vertex's neighbors. The Heffter-Edmonds principle states that, up to equivalence, rotation systems are in one-to-one correspondence with cellular embeddings in orientable surfaces. To determine the embedded surface from the rotation system, one can trace out its faces and solve for the genus in Euler's polyhedral formula.

For a simple graph on at least three vertices, every face has length at least 3, so the inequality $2|E| \geq 3|F|$ yields the bound
$$|E| \leq 3|V|-6+6g,$$
with equality if and only if the embedding is triangular. Let $\gamma(G)$ denote the \emph{minimum genus} of $G$, i.e., the minimum value $\gamma$ such that $G$ has an embedding in the orientable surface of genus $\gamma$. The upper bound on the number of edges implies a lower bound on the minimum genus of simple graphs, which, for the complete graphs $K_n$, reads
$$\gamma(K_n) \geq \left \lceil \frac{(n-3)(n-4)}{12} \right\rceil.$$
The Map Color Theorem~\cite{Ringel-MapColor} of Ringel, Youngs, and others exhibits embeddings matching this lower bound for all $n \geq 3$. In this work, we deal with embeddings whose genus is close to, but not exactly equal to this bound. To this end, we define two parameters that measure the distance away from this bound. 

If $f_1, f_2, \dotsc f_{|F|}$ are the lengths of the faces of the embedding, then we define the \emph{face excess} to be
$$f^+ = \sum_{i=1}^{|F|} (f_i-3),$$
which roughly describes how far the embedding is from triangular. We also define the \emph{vertex excess} of a $c$-connected graph to be $|V|-(c+1)$, i.e., how many more vertices it has than the complete graph of the  connectivity $c$. A standard calculation using Euler's polyhedral formula yields:

\begin{lemma}[Lemma 2, Brinkmann et al.~\cite{BNV}]
Suppose we have an embedding of a graph with minimum degree $c \geq 6$ and vertex and face excesses $v^+$ and $f^+$. Then the genus of this embedding is at least
$$\left\lceil \frac{(c-2)(c-3)}{12} + \frac{(c-6)v^+}{12} + \frac{f^+}{6} \right\rceil \geq \gamma(K_{c+1}) + \left\lfloor \frac{(c-6)v^+}{12}+\frac{f^+}{6}\right\rfloor.$$
\label{lem-euler}
\end{lemma}

Previously, this type of calculation was used by Plummer and Zha~\cite{PlummerZha} in determining what they called ``$g$-unique'' graphs:

\begin{corollary}[Theorem 2.4(a), Plummer and Zha~\cite{PlummerZha}]
For $c \geq 18$, no other $c$-connected graphs besides the complete graphs can have genus $\gamma(K_{c+1})$. 
\end{corollary}

In conjunction with the main result of Brinkmann et al.~\cite{BNV}, this implies a lower bound on the genus of any dual-separable embedding of a graph of sufficiently high connectivity. We call a cutvertex in the dual graph a \emph{cutface}. 

\begin{theorem}[Theorem 2, Brinkmann et al.~\cite{BNV}]
Suppose we have a dual-separable embedding of a graph of vertex connectivity $c \geq 8$. Then the embedding's cutface is of length at least 15.
\end{theorem}
\begin{corollary}[Corollary 3, Brinkmann et al.~\cite{BNV}]
For $c \geq 8$, $\delta_1(c) \geq \gamma(K_{c+1})+2$.
\end{corollary}
\begin{proof}
The face excess is at least 12, so substituting this into the right-hand side of Lemma \ref{lem-euler} yields the desired result. 
\end{proof}

Our construction results in dual-separable embeddings with a cutface of length 18. By the left-hand side of Lemma \ref{lem-euler}, they can only be optimal if $c \equiv 0, 1, 4, 5, 8, 9 \pmod{12}$. We focus on the case where $c \equiv 4 \pmod{12}$. 

\section{Current graphs}

A \emph{current graph} consists of an embedding $\phi\colon G \to S$ in an orientable surface and an arc-labeling $\alpha\colon E(G)^+ \to \Gamma$, where $\Gamma$ is an abelian group and the labeling satisfies $\alpha(e^+) = -\alpha(e^-)$ for every edge $e \in E(G)$. The labels are called \emph{currents} and $\Gamma$ is called the \emph{current group}. The \emph{excess} of a vertex is the sum of the currents on the arcs directed towards the vertex, and if the excess is 0, we say that the vertex satisfies \emph{Kirchhoff's current law}. Each face boundary walk is called a \emph{circuit}, and it consists of a cyclic sequence of arcs $(e^\pm_1, e^\pm_2, e^\pm_3, \dotsc)$ with a consistent orientation. 

We focus our attention to \emph{index 3} current graphs, those where the embeddings have exactly three faces. Following Section 9 of Ringel~\cite{Ringel-MapColor}, the index 3 current graphs we consider have current group $\mathbb{Z}_{12s+3}$ and satisfy the following standard properties:

\begin{itemize}
\item Each vertex has degree 3.
\item The embedding has exactly three circuits labeled $[0]$, $[1]$, and $[2]$.
\item The log of each circuit contains each element of $\mathbb{Z}_{12s+3}\setminus\{0\}$ exactly once.
\item Unlabeled vertices satisfy Kirchhoff's current law.
\item For each labeled vertex, each circuit is incident with the vertex, and the excess of the vertex generates the index 3 subgroup of $\mathbb{Z}_{12s+3}$.
\item For each edge, if circuit $[a]$ traverses $e^+$ and circuit $[b]$ traverses $e^-$, then $\alpha(e^+) \equiv b-a \pmod{3}$. 
\end{itemize}

The derived embedding of the current graph is generated in the following way. The \emph{log} of a circuit is another cyclic sequence that replaces each arc $e^\pm_i$ with its label $\alpha(e^\pm_i)$. The derived graph initially has the vertex set $\mathbb{Z}_{12s{+}3}$, where the rotation at each vertex $j \in \mathbb{Z}_{12s{+}3}$ is found by taking the log of circuit $[j \bmod{3}]$ and adding $j$ to each element. Each labeled vertex induces a Hamiltonian $(12s+3)$-sided face. We subdivide each such face with a new vertex and connect it to every vertex incident with that face to obtain a triangular embedding of $K_{12s+3+\ell}-E(K_\ell)$, where $\ell$ is the number of labeled vertices.

\section{An optimal construction}

Jungerman and Ringel~\cite{JungermanRingel-Minimal} described a substructure in a triangular embedding called a \emph{subtractible handle}, a set of six edges whose removal from a rotation system decreases the genus by 1 and keeps the embedding triangular. One can leverage a subtractible handle to manufacture a cutvertex in the dual graph:

\begin{proof}[Proof of Theorem \ref{thm-main}]
For $s = 1$, refer to the embedding of the complete graph $K_{17}$ in Appendix \ref{app-small}. For $s \geq 2$, we show that $K_{12s+5}$ has an optimal dual-separable embedding by repurposing a family of current graphs in Sun~\cite{Sun-Minimum}, shown in Figure \ref{fig-case5}. It is derived from the family presented in Chapter 9.2 of Ringel~\cite{Ringel-MapColor} by swapping the ``rungs'' containing the currents labeled $6$ and $12s-3$. 

The current graphs generate triangular embeddings of $K_{12s+5}-E(K_2)$ in the surface of genus $\gamma(K_{12s+5})-1$. By examining the part of the current graphs near the two labeled vertices, we find that each rotation system is of the form:
    
$$\begin{array}{rrrccccll}
0. & \dotsc & 12s{-}4 & 5 & 9 & 4 & 6 & 2 & \dotsc \\
3. & \dotsc & 7 & 9 & 5 & 6 & 4 & y & \dotsc \\
4. & \dotsc & y & 3 & 6 & 0 & 9 & 16 & \dotsc \\
5. & \dotsc & x & 6 & 3 & 9 & 0 & 12s{-}4 & \dotsc \\
6. & \dotsc & 2 & 0 & 4 & 3 & 5 & x & \dotsc \\
9. & \dotsc & 16 & 4 & 0 & 5 & 3 & 7 & \dotsc \\
\end{array}$$

\begin{figure}[h]
\centering
\includegraphics[scale=0.95]{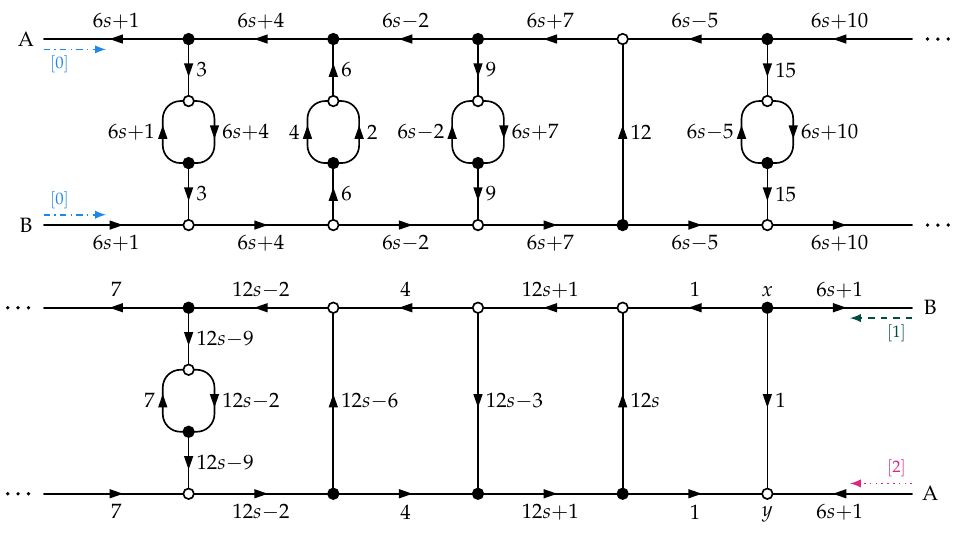}
\caption{A family of index 3 current graphs with current group $\Z_{12s+3}$, $s \geq 2$.}
\label{fig-case5}
\end{figure}

The middle four columns form the so-called subtractible handle. Deleting the edges corresponding to the middle two columns of the above table, shown visually in Figure \ref{fig-mod}(a), decreases the genus of the embedding by 1 and reveals two new triangular faces $[0,6,5]$ and $[3,4,9]$. We ``reverse'' the triangles, as in Figure \ref{fig-mod}(b): in the rotation of each vertex incident with one of the two triangles, we transpose the other two vertices of the triangle:

$$\begin{array}{rrrccll}
0. & \dotsc & 12s{-}4 & 6 & 5 & 2 & \dotsc \\
3. & \dotsc & 7 & 4 & 9 & y & \dotsc \\
4. & \dotsc & y & 9 & 3 & 16 & \dotsc \\
5. & \dotsc & x & 0 & 6 & 12s{-}4 & \dotsc \\
6. & \dotsc & 2 & 5 & 0 & x & \dotsc \\
9. & \dotsc & 16 & 3 & 4 & 7 & \dotsc \\
\end{array}$$

\begin{figure}[h]
\centering
\includegraphics[scale=1]{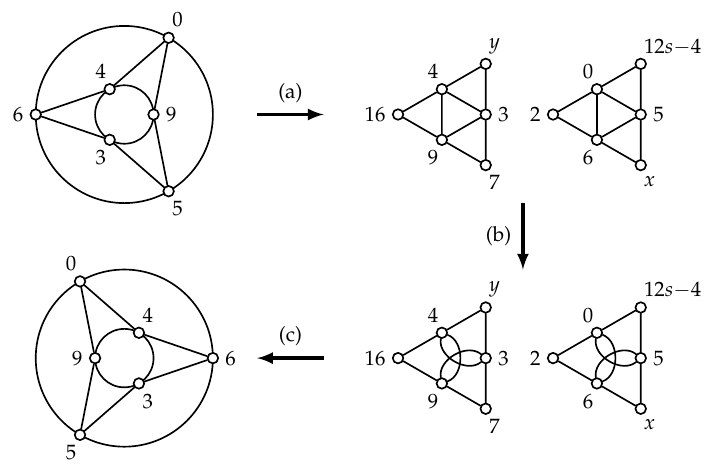}
\caption{Reversing a subtractible handle.}
\label{fig-mod}
\end{figure}

The three triangular faces neighboring a reversed triangle become merged together into a $9$-gon. The total number of faces decreases by 4, so the genus increases by 2. If we add back the edges of the subtractible handle (in the opposite orientation) like in Figure \ref{fig-mod}(c), the two 9-gons form a cut of size 2 in the dual graph, separating the six faces of the subtractible handle from the rest of the faces. Finally, we connect the two $9$-gons with a handle and add three edges between the two $9$-gons to avoid introducing any self-loops or parallel edges in the dual graph. The missing edge $(x,y)$ is chosen to be one of those three edges, as seen in Figure \ref{fig-handle}. The other two edges must already be in the graph, so we delete them from their original locations. 

\begin{figure}[h]
\centering
\includegraphics[scale=1]{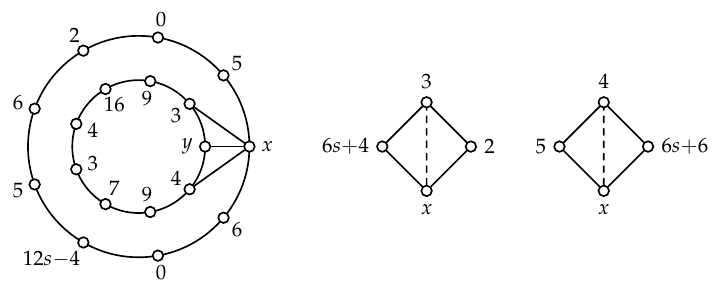}
\caption{Joining two faces into one large face.}
\label{fig-handle}
\end{figure}

The resulting $18$-gon is a cutvertex in the dual, so it remains to check that the dual is simple. Since the graph has minimum degree at least 3, faces of length 3 or 4 cannot have self-incidences or share two edges with one another. Thus, we only need to verify that the $18$-gon is not incident with itself or the same face twice. It is incident with each face belonging to the subtractible handle once, and the other incidences are with the faces neighboring the ones shown in Figure \ref{fig-mod}(b). Since the minimum degree is at least 4, each of the triangular neighboring faces are distinct (e.g., $[4, y, 16]$ cannot be a face). Finally, one can directly check that there is exactly once incidence with the quadrangular face $[4, 5, x, 6s+6]$,

In total, the genus increased by 3, so the final embedding has genus $\gamma(K_{12s+5})+2$.
\end{proof}

\section{Conclusion}

We exhibited optimal dual-separable embeddings of complete graphs $K_{c+1}$ for small values of $c$ and the infinite family $c = 12s+4$ for all $s \geq 2$. However, we were unable to resolve the case where $c = 11$. By Lemma \ref{lem-euler}, for $c = 11, 14, 15$, $K_{c+1}$ is the only $c$-connected graph that could possibly have an optimal dual-separable embedding, and furthermore, any such embedding must have all triangular faces except exactly one 15-sided face. We were successful in finding such an embedding for $c = 14,15$, but not $c = 11$ (though the author is unsure if the search was exhaustive). Additionally, Brinkmann et al.~\cite{BNV} found optimal embeddings for $c = 8,9$. These results lead us to believe the following:

\begin{conjecture}
For $c \geq 8$, $\delta_1(c) = \gamma(K_{c+1})+2$, except when $c = 11$. 
\end{conjecture}

In the construction in Theorem \ref{thm-main}, three additional edges are needed to merge the two 9-sided faces together while keeping the embedding cellular and the dual graph simple, and one of those edges was used to connect two nonadjacent vertices. In principle, one could apply this construction to other complete graphs $K_{c+1}$, where deleting up to three edges can cause the genus to decrease by 1. As mentioned earlier, this is feasible for $n \equiv 0, 1, 4, 5, 8, 9 \pmod{12}$, but no approaches are known for the other residues. 

Instead of starting with an embedding of a near-complete graph, one could also apply the construction directly to embeddings of $K_n$. For sufficiently large $n$, there are genus embeddings of $K_n$ with subtractible handles: the necessary current graphs (possibly with a similar ``rung-swapping'' modification) for most residues are already presented in Ringel~\cite{Ringel-MapColor}, while the $n \equiv 0, 8 \pmod{12}$ cases require the newer families of current graphs in Sun~\cite{Sun-K12s, Sun-FaceDist}. Since the construction in the proof of Theorem \ref{thm-main} increases the genus by 3, this would equal the nearly optimal upper bound of $\gamma(K_n)+3$ given by Brinkmann et al.~\cite{BNV} for $n \equiv 9 \pmod{12}$. One would need to tediously check that all of the resulting embeddings have simple duals. 

\bibliographystyle{alpha}
\bibliography{biblio}

\begin{thebibliography}{BNVdC23}

\bibitem[BBZ22]{BBZ}
Drago Bokal, Gunnar Brinkmann, and Carol~T Zamfirescu.
\newblock The connectivity of the dual.
\newblock {\em Journal of Graph Theory}, 101(2):182--209, 2022.

\bibitem[BNVdC23]{BNV}
Gunnar Brinkmann, Kenta Noguchi, and Heidi Van~den Camp.
\newblock Face sizes and the connectivity of the dual.
\newblock {\em arXiv preprint arXiv:2309.17121}, 2023.

\bibitem[GT87]{GrossTucker}
Jonathan~L. Gross and Thomas~W. Tucker.
\newblock {\em {Topological Graph Theory}}.
\newblock John Wiley \& Sons, 1987.

\bibitem[JR80]{JungermanRingel-Minimal}
Mark Jungerman and Gerhard Ringel.
\newblock Minimal triangulations on orientable surfaces.
\newblock {\em Acta Mathematica}, 145(1):121--154, 1980.

\bibitem[Jun74]{Jungerman-K18}
Mark Jungerman.
\newblock {Orientable triangular embeddings of $K_{18}-K_3$ and $K_{13}-K_3$}.
\newblock {\em Journal of Combinatorial Theory, Series B}, 16(3):293--294,
  1974.

\bibitem[PZ98]{PlummerZha}
Michael~D. Plummer and Xiaoya Zha.
\newblock On the connectivity of graphs embedded in surfaces.
\newblock {\em Journal of Combinatorial Theory, Series B}, 72(2):208--228,
  1998.

\bibitem[Rin74]{Ringel-MapColor}
Gerhard Ringel.
\newblock {\em {Map Color Theorem}}.
\newblock Springer Science \& Business Media, 1974.

\bibitem[Sun19]{Sun-K12s}
Timothy Sun.
\newblock A simple construction for orientable triangular embeddings of the
  complete graphs on $12s$ vertices.
\newblock {\em Discrete Mathematics}, 342(4):1147--1151, 2019.

\bibitem[Sun20]{Sun-Minimum}
Timothy Sun.
\newblock Simultaneous current graph constructions for minimum triangulations
  and complete graph embeddings.
\newblock {\em Ars Mathematica Contemporanea}, 18(2):309--337, 2020.

\bibitem[Sun21]{Sun-FaceDist}
Timothy Sun.
\newblock Face distributions of embeddings of complete graphs.
\newblock {\em Journal of Graph Theory}, 97(2):281--304, 2021.

\bibitem[Sun24]{Sun-Kainen}
Timothy Sun.
\newblock On {K}ainen's conjectures on surface crossing numbers.
\newblock {\em arXiv preprint arXiv:2405.06118}, 2024.

\end{thebibliography}

\appendix

\section{Some small embeddings}\label{app-small}

Using computer search, we found optimal dual-separable embeddings of the complete graphs on 11, 13, 14, 15, 16, and 17 vertices. In each case, we highlight the corners of the cutface and note that they always surround the same set of six faces involving the vertices $0, \dotsc, 5$. In particular, the cutface separates those six faces from all of the other faces.

$K_{11}$ with a cutface of length 17:

$$\SM\begin{array}{rccccccccccccccccccccccccc}
0. & \RL{6} & \RL{1} & 3 & 5 & \RL{2} & \RL{9} & 8 & 10 & 4 & 7 \\
1. & \RL{6} & \RL{2} & 4 & 3 & \RL{0} & \RL{5} & 10 & 7 & 8 & 9 \\
2. & \RL{8} & \RL{0} & 5 & 4 & \RL{1} & \RL{7} & 10 & 6 & 9 & 3 \\
3. & \RL{7} & \RL{5} & 0 & 1 & \RL{4} & \RL{6} & 10 & 8 & 2 & 9 \\
4. & \RL{9} & \RL{3} & 1 & 2 & \RL{5} & \RL{6} & 8 & 7 & 0 & 10  \\
5. & \RL{1} & \RL{4} & 2 & 0 & \RL{3} & \RL{8} & 6 & 7 & 9 & 10 \\
6. & \RL{3} & \RL{0} & 7 & 5 & 8 & \RL{4} & \RL{1} & 9 & 2 & 10 \\
7. & \RL{2} & \RL{3} & 9 & 5 & 6 & 0 & 4 & 8 & 1 & 10 \\
8. & \RL{2} & \RL{3} & 10 & 0 & 9 & 1 & 7 & 4 & 6 & 5 \\
9. & \RL{0} & \RL{4} & 10 & 5 & 7 & 3 & 2 & 6 & 1 & 8 \\
10. & 0 & 8 & 3 & 6 & 2 & 7 & 1 & 5 & 9 & 4 \\
\end{array}$$

$K_{13}$ with a cutface of length 15 and three quadrangular faces:

$$\SM\begin{array}{rccccccccccccccccccccccccc}
0. & \RL{4} & \RL{1} & 3 & 5 & \RL{2} & \RL{8} & 7 & 10 & 6 & 9 & 11 & 12 \\
1. & \RL{5} & \RL{2} & 4 & 3 & \RL{0} & \RL{6} & 7 & 12 & 11 & 8 & 9 & 10 \\
2. & \RL{3} & \RL{0} & 5 & 4 & \RL{1} & \RL{7} & 8 & 12 & 10 & 9 & 6 & 11 \\
3. & \RL{6} & \RL{5} & 0 & 1 & \RL{4} & \RL{2} & 11 & 9 & 7 & 8 & 10 & 12 \\
4. & \RL{7} & \RL{3} & 1 & 2 & \RL{5} & \RL{0} & 12 & 8 & 9 & 6 & 10 & 11 \\
5. & \RL{8} & \RL{4} & 2 & 0 & \RL{3} & \RL{1} & 10 & 7 & 11 & 6 & 12 & 9 \\
6. & \RL{1} & \RL{3} & 12 & 5 & 11 & 2 & 0 & 10 & 4 & 9 & 8 & 7 \\
7. & \RL{2} & \RL{4} & 11 & 5 & 10 & 0 & 3 & 9 & 12 & 1 & 6 & 8 \\
8. & \RL{0} & \RL{5} & 4 & 12 & 2 & 7 & 6 & 9 & 1 & 11 & 10 & 3 \\
9. & 0 & 2 & 10 & 1 & 8 & 6 & 4 & 5 & 12 & 7 & 3 & 11 \\
10. & 0 & 7 & 5 & 1 & 9 & 2 & 12 & 3 & 8 & 11 & 4 & 6 \\
11. & 0 & 9 & 3 & 2 & 6 & 5 & 7 & 4 & 10 & 8 & 1 & 12 \\
12. & 0 & 11 & 1 & 7 & 9 & 5 & 6 & 3 & 10 & 2 & 8 & 4 \\
\end{array}$$

$K_{14}$ with a cutface of length $17$ and a hexagonal face:

$$\SM\begin{array}{rccccccccccccccccccccccccc}
0. & \RL{6} & \RL{1} & 3 & 5 & \RL{2} & \RL{4} & 7 & 13 & 12 & 11 & 9 & 10 & 8 \\
1. & \RL{6} & \RL{2} & 4 & 3 & \RL{0} & \RL{5} & 13 & 8 & 7 & 9 & 11 & 10 & 12 \\
2. & \RL{7} & \RL{0} & 5 & 4 & \RL{1} & \RL{3} & 11 & 13 & 6 & 8 & 12 & 9 & 10 \\
3. & \RL{2} & \RL{5} & 0 & 1 & \RL{4} & \RL{8} & 13 & 9 & 7 & 12 & 10 & 6 & 11 \\
4. & \RL{0} & \RL{3} & 1 & 2 & \RL{5} & \RL{6} & 10 & 13 & 11 & 8 & 9 & 12 & 7 \\
5. & \RL{1} & \RL{4} & 2 & 0 & \RL{3} & \RL{7} & 8 & 10 & 11 & 9 & 6 & 12 & 13 \\
6. & \RL{4} & \RL{1} & 12 & 5 & \RL{9} & \RL{0} & 8 & 2 & 13 & 7 & 11 & 3 & 10 \\
7. & \RL{5} & \RL{2} & 10 & 11 & 6 & 13 & 0 & 4 & 12 & 3 & 9 & 1 & 8 \\
8. & \RL{3} & \RL{9} & 4 & 11 & 12 & 2 & 6 & 0 & 10 & 5 & 7 & 1 & 13 \\
9. & \RL{8} & \RL{6} & 5 & 11 & 1 & 7 & 3 & 13 & 10 & 0 & 2 & 12 & 4 \\
10. & 0 & 9 & 13 & 4 & 6 & 3 & 12 & 1 & 11 & 7 & 2 & 5 & 8 \\
11. & 0 & 12 & 8 & 4 & 13 & 2 & 3 & 6 & 7 & 10 & 1 & 9 & 5 \\
12. & 0 & 13 & 5 & 6 & 1 & 10 & 3 & 7 & 4 & 9 & 2 & 8 & 11 \\
13. & 0 & 7 & 6 & 2 & 11 & 4 & 10 & 9 & 3 & 8 & 1 & 5 & 12
\end{array}$$

$K_{15}$ with a cutface of length $15$:

$$\SM\begin{array}{rccccccccccccccccccccccccc}
0. & \RL{8} & \RL{1} & 3 & 5 & \RL{2} & \RL{4} & 10 & 13 & 12 & 9 & 6 & 11 & 14 & 7 \\
1. & \RL{6} & \RL{2} & 4 & 3 & \RL{0} & \RL{5} & 13 & 7 & 9 & 14 & 8 & 12 & 10 & 11 \\
2. & \RL{7} & \RL{0} & 5 & 4 & \RL{1} & \RL{3} & 6 & 10 & 12 & 14 & 9 & 8 & 13 & 11 \\
3. & \RL{2} & \RL{5} & 0 & 1 & \RL{4} & \RL{8} & 9 & 11 & 13 & 10 & 14 & 12 & 7 & 6 \\
4. & \RL{0} & \RL{3} & 1 & 2 & \RL{5} & \RL{6} & 14 & 11 & 8 & 7 & 12 & 13 & 9 & 10 \\
5. & \RL{1} & \RL{4} & 2 & 0 & \RL{3} & \RL{7} & 14 & 10 & 8 & 11 & 12 & 6 & 9 & 13 \\
6. & \RL{4} & \RL{1} & 11 & 0 & 9 & 5 & 12 & 8 & 10 & 2 & 3 & 7 & 13 & 14 \\
7. & \RL{5} & \RL{2} & 11 & 10 & 9 & 1 & 13 & 6 & 3 & 12 & 4 & 8 & 0 & 14 \\
8. & \RL{3} & \RL{0} & 7 & 4 & 11 & 5 & 10 & 6 & 12 & 1 & 14 & 13 & 2 & 9 \\
9. & 0 & 12 & 11 & 3 & 8 & 2 & 14 & 1 & 7 & 10 & 4 & 13 & 5 & 6 \\
10. & 0 & 4 & 9 & 7 & 11 & 1 & 12 & 2 & 6 & 8 & 5 & 14 & 3 & 13 \\
11. & 0 & 6 & 1 & 10 & 7 & 2 & 13 & 3 & 9 & 12 & 5 & 8 & 4 & 14 \\
12. & 0 & 13 & 4 & 7 & 3 & 14 & 2 & 10 & 1 & 8 & 6 & 5 & 11 & 9 \\
13. & 0 & 10 & 3 & 11 & 2 & 8 & 14 & 6 & 7 & 1 & 5 & 9 & 4 & 12 \\
14. & 0 & 11 & 4 & 6 & 13 & 8 & 1 & 9 & 2 & 12 & 3 & 10 & 5 & 7
\end{array}$$

$K_{16}$ with a cutface of length $15$:

$$\SM\begin{array}{rccccccccccccccccccccccccc}
0. & \RL{8} & \RL{1} & 3 & 5 & \RL{2} & \RL{4} & 13 & 14 & 10 & 6 & 11 & 7 & 12 & 9 & 15 \\
1. & \RL{6} & \RL{2} & 4 & 3 & \RL{0} & \RL{5} & 10 & 8 & 13 & 9 & 12 & 15 & 11 & 14 & 7 \\
2. & \RL{7} & \RL{0} & 5 & 4 & \RL{1} & \RL{3} & 9 & 13 & 11 & 15 & 14 & 6 & 12 & 8 & 10 \\
3. & \RL{2} & \RL{5} & 0 & 1 & \RL{4} & \RL{8} & 14 & 11 & 12 & 6 & 13 & 10 & 15 & 7 & 9 \\
4. & \RL{0} & \RL{3} & 1 & 2 & \RL{5} & \RL{6} & 10 & 11 & 9 & 14 & 12 & 7 & 8 & 15 & 13 \\
5. & \RL{1} & \RL{4} & 2 & 0 & \RL{3} & \RL{7} & 11 & 13 & 6 & 9 & 8 & 12 & 14 & 15 & 10 \\
6. & \RL{4} & \RL{1} & 7 & 15 & 9 & 5 & 13 & 3 & 12 & 2 & 14 & 8 & 11 & 0 & 10 \\
7. & \RL{5} & \RL{2} & 10 & 9 & 3 & 15 & 6 & 1 & 14 & 13 & 8 & 4 & 12 & 0 & 11 \\
8. & \RL{3} & \RL{0} & 15 & 4 & 7 & 13 & 1 & 10 & 2 & 12 & 5 & 9 & 11 & 6 & 14 \\
9. & 0 & 12 & 1 & 13 & 2 & 3 & 7 & 10 & 14 & 4 & 11 & 8 & 5 & 6 & 15 \\
10. & 0 & 14 & 9 & 7 & 2 & 8 & 1 & 5 & 15 & 3 & 13 & 12 & 11 & 4 & 6 \\
11. & 0 & 6 & 8 & 9 & 4 & 10 & 12 & 3 & 14 & 1 & 15 & 2 & 13 & 5 & 7 \\
12. & 0 & 7 & 4 & 14 & 5 & 8 & 2 & 6 & 3 & 11 & 10 & 13 & 15 & 1 & 9 \\
13. & 0 & 4 & 15 & 12 & 10 & 3 & 6 & 5 & 11 & 2 & 9 & 1 & 8 & 7 & 14 \\
14. & 0 & 13 & 7 & 1 & 11 & 3 & 8 & 6 & 2 & 15 & 5 & 12 & 4 & 9 & 10 \\
15. & 0 & 9 & 6 & 7 & 3 & 10 & 5 & 14 & 2 & 11 & 1 & 12 & 13 & 4 & 8
\end{array}$$

$K_{17}$ with a cutface of length $17$ and a hexagonal face:

$$\SM\begin{array}{rccccccccccccccccccccccccc}
0. & \RL{6} & \RL{1} & 3 & 5 & \RL{2} & \RL{4} & 14 & 10 & 13 & 12 & 11 & 9 & 16 & 15 & 7 & 8 \\
1. & \RL{6} & \RL{2} & 4 & 3 & \RL{0} & \RL{5} & 11 & 15 & 14 & 8 & 7 & 10 & 9 & 12 & 16 & 13 \\
2. & \RL{7} & \RL{0} & 5 & 4 & \RL{1} & \RL{3} & 12 & 6 & 16 & 8 & 9 & 10 & 15 & 13 & 14 & 11 \\
3. & \RL{2} & \RL{5} & 0 & 1 & \RL{4} & \RL{8} & 11 & 9 & 6 & 14 & 13 & 16 & 7 & 15 & 10 & 12 \\
4. & \RL{0} & \RL{3} & 1 & 2 & \RL{5} & \RL{6} & 15 & 16 & 12 & 7 & 9 & 10 & 8 & 13 & 11 & 14 \\
5. & \RL{1} & \RL{4} & 2 & 0 & \RL{3} & \RL{7} & 12 & 13 & 8 & 16 & 9 & 14 & 15 & 6 & 10 & 11 \\
6. & \RL{4} & \RL{1} & 13 & 7 & 11 & 16 & 2 & 12 & 14 & 3 & \RL{9} & \RL{0} & 8 & 10 & 5 & 15 \\
7. & \RL{5} & \RL{2} & 11 & 6 & 13 & 10 & 1 & 8 & 0 & 15 & 3 & 16 & 14 & 9 & 4 & 12 \\
8. & \RL{3} & \RL{9} & 2 & 16 & 5 & 13 & 4 & 10 & 6 & 0 & 7 & 1 & 14 & 12 & 15 & 11 \\
9. & \RL{8} & \RL{6} & 3 & 11 & 13 & 15 & 12 & 1 & 10 & 4 & 7 & 14 & 5 & 16 & 0 & 2 \\
10. & 0 & 14 & 16 & 11 & 12 & 3 & 15 & 2 & 5 & 6 & 8 & 4 & 9 & 1 & 7 & 13 \\
11. & 0 & 12 & 10 & 16 & 6 & 7 & 2 & 14 & 4 & 13 & 9 & 3 & 8 & 15 & 1 & 5 \\
12. & 0 & 13 & 5 & 7 & 4 & 16 & 1 & 9 & 15 & 8 & 14 & 6 & 2 & 3 & 10 & 11 \\
13. & 0 & 10 & 7 & 6 & 1 & 16 & 3 & 14 & 2 & 15 & 9 & 11 & 4 & 8 & 5 & 12 \\
14. & 0 & 4 & 11 & 2 & 13 & 3 & 6 & 12 & 8 & 1 & 15 & 5 & 9 & 7 & 16 & 10 \\
15. & 0 & 16 & 4 & 6 & 5 & 14 & 1 & 11 & 8 & 12 & 9 & 13 & 2 & 10 & 3 & 7 \\
16. & 0 & 9 & 5 & 8 & 2 & 6 & 11 & 10 & 14 & 7 & 3 & 13 & 1 & 12 & 4 & 15
\end{array}$$

\end{document}